\documentclass[11pt,twoside]{amsart}
\usepackage{amsmath, amsthm, amscd, amsfonts, amssymb, graphicx, color}
\usepackage[bookmarksnumbered, colorlinks, plainpages]{hyperref}
\usepackage{pgf,tikz}
\usetikzlibrary{arrows}
\usepackage[utf8]{inputenc}
\usepackage{pstricks-add}
\input{mathrsfs.sty}
\newtheorem{thm}{Theorem}[section]

\newtheorem{lem}[thm]{Lemma}
\newtheorem{prop}[thm]{Proposition}

\newtheorem{rem}[thm]{\bf{Remark}}
\newtheorem{example}[thm]{\bf{Example}}

\numberwithin{equation}{section}

\begin{document}

%------------------------------------------------------------------------------------%

\title{Unmixed $r$-partite Graphs }
\author{Reza jafarpour-Golzari and Rashid Zaare-Nahandi}
\address{Department of Mathematics, Institute for Advanced Studies
in Basic Science (IASBS), P.O.Box 45195-1159, Zanjan, Iran}

\email{r.golzary@iasbs.ac.ir}

\address{Department of Mathematics, Institute for Advanced Studies
in Basic Science (IASBS), P.O.Box 45195-1159, Zanjan, Iran}

\email{rashidzn@iasbs.ac.ir}

\thanks{{\scriptsize
\hskip -0.4 true cm MSC(2010): Primary: 5E40; Secondary: 5C69, 5C75.
\newline Keywords: $r$-partite graph, well-covered, unmixed,
perfect matching, clique.\\
\\
%$*$Corresponding author
\newline\indent{\scriptsize}}}

\maketitle

%------------------------------------------------------------------------------------%

\begin{abstract}
Unmixed bipartite graphs have been characterized by Ravadra and
Villarreal independently. Our aim in this paper is to
characterize unmixed  $r$-partite graphs under a certain condition,
witch is a generalization of villarreal's theorem on bipartite
graphs. Also we give some examples and counterexamples in relevance this subject.
\end{abstract}

\vskip 0.2 true cm

%------------------------------------------------------------------------------------%

\section{\bf Introduction}
\vskip 0.4 true cm

In the sequel, we use \cite{EV} as reference for terminology and
notation on graph theory.

Let $G$ be a simple finite graph with vertex set $V(G)$ and edge set $E(G)$.
 A subset $C$  of $V(G)$ is said to be a vertex cover of G if every
edge of $G$, is adjacent with some vertices in $C$. A vertex cover
$C$ is called minimal, if there is no proper subset of $C$ which is a vertex cover. A graph is called unmixed, if all minimal
vertex covers of $G$ have the same number of elements. A subset $H$
of $V(G)$ is said to be independent, if $G$ has not
any edge $\{x, y\}$ such that $\{x, y\}\subseteq H$. A maximal
independent set of $G$, is an independent set $I$ of $G$, such
that for every $H\supsetneqq I$, $H$ is not an independent set of
$G$. Notice that $C$ is a minimal vertex cover if and only if
$V(G)\setminus C$ is a maximal independent set. A graph $G$ is called
well-covered if all the maximal independent sets of $G$ have the
same cardinality. Therefore a graph is unmixed if and only if it
is well-covered. The minimum cardinality of all minimal vertex
covers of $G$ is called the covering number of $G$, and the
maximum cardinality of all maximal independent sets of G is
called the independence number of $G$. For determining the
independence number see \cite{K}. For relation between unmixedness of a graph and other graph properties see \cite{EE, HH, V, Z}.

Well-covered graphs were introduced by Plummer.
See \cite{P} for a survey on well-covered graphs and properties
of them. For an integer $r\geq 2$, a graph $G$ is said to
be $r$-partite, if $V(G)$ can be partitioned into $r$
disjoint parts such that for every $\{x, y\}\in E(G)$, $x$ and
$y$ do not lie in the same part. If $r=2, 3$, $G$ is
said to be bipartite and tripartite, respectively. Let G be an $r$-partite graph. For a vertex $v\in V(G)$, let $N(v)$ be the set of all vertices $u\in
V(G)$ where $\{u, v\}$ be an edge of $G$. Let $G$ be a bipartite
graph, and let $e=\{u, v\}$ be an edge of $G$. Then $G_{e}$ is the
subgraph induced on $N(u)\cup N(v)$. If $G$ is connected, the distance between $x$ and $y$ where $x,
y\in V(G)$, denoted by $d(x, y)$, is the length of the shortest
path between $x$ and $y$. A set $M\subseteq E(G)$ is said to be a
matching of $G$, if for any two $\{x, y\}, \{x', y'\}\in M$,
$\{x, y\}\cap \{x', y'\}= \emptyset$. A matching $M$ of $G$ is called
perfect if for every $v\in V(G)$, there exists an edge $\{x,
y\}\in M$ such that $v\in \{x, y\}$. A clique in G is a set Q of
vertices such that for every $x, y\in Q$, if $x\neq y$, $x, y$
lie in an edge. An $r$-clique is a clique of size r.

Unmixed bipartite graphs have already been characterized by
Ravindra and villarreal in a combinatorial way independently
\cite{R, V2}. Also these graphs have been characterize in an
algebraic method \cite{V1}.

\smallskip

In 1977, Ravindra gave the following criteria for unmixedness of
bipartite graphs.

\smallskip

\begin{thm}
\cite{R} Let $G$ be a connected bipartite graph. Then $G$ is
unmixed if and only if $G$ contains a perfect matching $F$ such
that for every edge $e=\{x, y\}\in F$, the induced subgraph
$G_{e}$ is a complete bipartite graph.
\end{thm}

\smallskip

Villarreal in 2007, gave the following characterization of
unmixed bipartite graphs.

\smallskip

\begin{thm}
\cite[Theorem 1.1]{V2} Let $G$ be a bipartite graph without
isolated vertices. Then $G$ is unmixed if and only if there is a
bipartition $V_{1}=\{x_{1}, \ldots , x_{g}\}, V_{2}=\{y_{1},
\ldots , y_{g}\}$ of $G$ such that: (a) $\{x_{i}, y_{i}\}\in
E(G)$, for all i, and (b) if $\{x_{i}, y_{j}\}$ and $\{x_{j},
y_{k}\}$ are in $E(G)$, and $i, j, k$ are distinct, then
$\{x_{i}, y_{k}\}\in E(G)$.
\end{thm}

\smallskip

H. Haghighi in \cite{Ha} gives the following characterization
of unmixed tripartite graphs under certain conditions.

\smallskip

\begin{thm}

\cite[Theorem 3.2]{Ha} Let $G$ be a tripartite graph which satisfies the condition $(\ast )$. Then the graph $G$ is unmixed
if and only if the following conditions hold:

(1) If $\{u_{i}, x_{q}\}, \{v_{j}, y_{q}\}, \{w_{k}, z_{q}\}\in E(G)$, where no two vertices of $\{x_{q}, y_{q}, z_{q}\}$ lie in one of the tree parts of $V(G)$ and $i, j, k, q$ are distinct, then the set $\{u_{i}, v_{j}, w_{k}\}$ contains an edge of $G$.

(2) If $\{r, x_{q}\}, \{s, y_{q}\}, \{t, z_{q}\}$ are edges of $G$, where $r$ and $S$ belong to one of the three parts of $V(G)$ and $t$ belongs to another part, then the set $\{r, s, t\}$ contains an edge of $G$(here $r$ and $s$ may be equal).

\end{thm}

In the above theorem, he has considered the condition $(\ast )$ as:

being a tripartite graph with partitions
\[U=\{u_{1}, \ldots u_{n}\}, V=\{v_{1}, \ldots v_{n}\}, W=\{w_{1}, \ldots w_{n}\},\]
in which $\{u_{i}, v_{i}\}, \{u_{i}, w_{i}\}, \{v_{i}, w_{i}\}\in E(G)$, for all $i=1, \ldots , n$.

Also to simplify the notations, he has used $\{x_{i}, y_{i}, z_{i}\}$ and $\{r_{i}, s_{i}, t_{i}\}$ as two permutations of $\{u_{i}, v_{i}, w_{i} \}$.

\

We give a characterization of unmixed $r$-partite graphs under certain condition which we name it $(\ast )$(see Theorem 2.3).

\

In both theorems 2.1 and 2.2 in an unmixed connected bipartite
graph, there is a perfect matching, with cardinality equal
to the cardinality of a minimal vertex cover, i.e. $\frac{|V(G)|}{2}$. An unmixed graph with $n$ vertices such that its
independence number is $\frac{n}{2}$, is said to be very well-covered. The unmixed
connected bipartite graphs are contained in the class of very well-covered
graphs. A characterization of very well-covered graphs is given in
\cite{ET}.

\smallskip

\section{\bf {\bf \em{\bf A generalization}}}
\vskip 0.4 true cm

By the following proposition, bipartition in connected bipartite graphs is unique.

\begin{prop}
Let $G$ be a connected bipartite graph with bipartition $\{A,
B\}$, and let $\{X, Y\}$ be any bipartition of $G$. Then $\{A,
B\}=\{X, Y\}$.
\begin{proof}
Let $x\in A$ be an arbitrary vertex of $G$. Then
$x\in X$ or $x\in Y$. without loss of generality let
$x$ be in $X$. Let $a\in A$.
then $d(x,a)$ is even. Then $a$ and $x$ are in the same part (of partition
$\{X,Y\}$). Then $A\subseteq X$, and by the same argument we have $X\subseteq A$. Therefore $A=X$, and
then $\{A, B\}=\{X, Y\}$.
\end{proof}
\end{prop}

\smallskip

The above fact for bipartite graphs, is not true in case of tripartite graphs, as shown in the following example.

%\begin{figure}
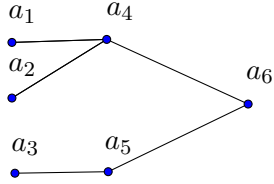
\definecolor{qqqqff}{rgb}{0.,0.,1.}
\begin{center}
\begin{tikzpicture}[line cap=round,line join=round,>=triangle 45,x=1.0cm,y=1.0cm]
\clip(0.,2.) rectangle (4.,5.);
\draw (1.8,4.44)-- (3.68,3.58);
\draw (1.8,4.44)-- (0.54,4.4);
\draw (1.8,4.44)-- (0.54,3.66);
\draw (0.54,3.66)-- (1.8,4.44);
\draw (1.8,4.44)-- (0.54,4.4);
\draw (3.68,3.58)-- (1.82,2.68);
\draw (1.82,2.68)-- (0.58,2.66);
\draw (0.36,5.02) node[anchor=north west] {$a_{1}$};
\draw (0.36,4.32) node[anchor=north west] {$a_{2}$};
\draw (0.4,3.26) node[anchor=north west] {$a_{3}$};
\draw (1.66,5.06) node[anchor=north west] {$a_{4}$};
\draw (1.64,3.28) node[anchor=north west] {$a_{5}$};
\draw (3.52,4.22) node[anchor=north west] {$a_{6}$};
\begin{scriptsize}
\draw [fill=qqqqff] (1.8,4.44) circle (1.5pt);
\draw [fill=qqqqff] (3.68,3.58) circle (1.5pt);
\draw [fill=qqqqff] (0.54,4.4) circle (1.5pt);
\draw [fill=qqqqff] (0.54,3.66) circle (1.5pt);
\draw [fill=qqqqff] (1.82,2.68) circle (1.5pt);
\draw [fill=qqqqff] (0.58,2.66) circle (1.5pt);
\end{scriptsize}
 \end{tikzpicture}

  \end{center}
 % \caption{}
 \label{graph 1}
% \end{figure}

\smallskip

In the above graph there are two different tripartitions:
  \[\{\{a_{1}, a_{2}, a_{3}\}, \{a_{4}, a_{5}\}, \{a_{6}\}\}\]
and
  \[\{\{a_{1}, a_{2}\}, \{a_{4}, a_{5}\}, \{ a_{3}, a_{6}\}\}.\]

\smallskip

A natural question refers to find criteria which
characterize a special class of unmixed $r$-partite $(r\geq 2)$
graphs.

\smallskip

In the above two characterizations of bipartite graphs, having a
perfect matching is essential in both proofs. This motivates us to impose the following condition.\\
\textit{We say a graph $G$ satisfies the condition $(\ast )$ for an integer $r\geq 2$, if $G$ can be partitioned to $r$ parts
 $V_{i}=\{x_{1i},
 \ldots , x_{ni}\}$,$(1\leq i\leq r)$, such that for all $1\leq j\leq n$, $\{x_{j1}, \ldots , x_{jr}\}$ is a clique.
 }

\smallskip

\begin{lem}
Let $G$ be a graph which satisfies $(\ast )$ for $r\geq 2$. If $G$ is unmixed,
then every minimal vertex cover of $G$, contains $(r-1)n$
vertices. Moreover the independence number of $G$ is
$n=\frac{|V(G)|}{r}$
\end{lem}
\begin{proof}
Let $C$ be a minimal vertex cover of $G$. Since for every $1\leq
j\leq n$, the vertices $x_{j1}, \ldots , x_{jr}$ are in a clique, $C$ must contain at least $r-1$ vertices in $\{x_{j1},
\ldots , x_{jr}\}$.
Therefore $C$ contains at least $(r-1)n$ vertices. By hypothesis
$\bigcup_{i=1}^{r-1} V_{i}$ is minimal vertex cover with $(r-1)n$
vertices, and $G$ is unmixed. Then every minimal vertex cover of
$G$ contains exactly $(r-1)n$ elements. The last claim can be
concluded from this fact that the complement of a minimal vertex
cover, is an independent set.
\end{proof}

Now we are ready for the main theorem.

\begin{thm}
Let $G$ be an $r$-partite graph which satisfies the
condition $(\ast )$ for $r$. Then $G$ is unmixed if and only if the
following condition hold:
\\For every $1\leq q\leq n$, if there is a set $\{x_{k_{1}s_{1}}, \ldots , x_{k_{r}s_{r}}\}$ such that
\[x_{k_{1}s_{1}}\thicksim x_{q1}, \ldots , x_{k_{r}s_{r}}\thicksim x_{qr},\]
 then the set $\{x_{k_{1}s_{1}}, \ldots , x_{k_{r}s_{r}}\}$ is not independent.
\end{thm}
\begin{proof}
Let $G$ be an arbitrary $r$-partite graph which satisfies the
condition $(\ast )$ for $r$.

Let $G$ be unmixed. We prove that mentioned condition holds. Assume the contrary.
Let
\[x_{k_{1}s_{1}}\thicksim x_{q1}, \ldots , x_{k_{r}s_{r}}\thicksim x_{qr},\]
but the set $\{x_{k_{1}s_{1}}, \ldots , x_{k_{r}s_{r}}\}$ is independent.
 Then there is a maximal independent set $M$, such that
$M$ contains this set. Since $M$ is maximal, $C=V(G)\backslash M$
is a minimal vertex cover of $G$. Since the set $\{x_{k_{1}s_{1}}, \ldots , x_{k_{r}s_{r}}\}$
is contained in $M$, then its elements are not in $C$, and since
$C$ is a cover of $G$, then all vertices $x_{qi}$, $(1\leq i\leq r)$ are in $C$. But
by Lemma 3.2, every minimal vertex cover, contains $n-1$ vertices
of clique $q$ th, a contradiction.

Conversely let the condition hold. We have to prove that $G$ is
unmixed. We show that all minimal vertex covers of $G$, intersect the set
$\{x_{q1}, \ldots, x_{qr}\}$ in exactly $r-1$ elements (for every
$1\leq q\leq n$). Let $C$ be a minimal vertex cover and $q$ be
arbitrary. Since $C$ is a vertex cover and $\{x_{q1}, \ldots,
x_{qr}\}$ is a clique, then $C$ intersects this set at least in
$r-1$ elements. Let the contrary. Let the cardinality of $C\cap
\{x_{q1}, \ldots, x_{qr}\}$ be $r$. Attending to minimality of
$C$, for every $1\leq i\leq r, N(x_{qi})$ contains at least one
element, distinct from the elements of $\{x_{q1}, \ldots,
x_{qr}\}\backslash\{x_{qi}\}$, which is not in $C$, because we
can not remove $x_{qi}$ of cover. Let this element be
$x_{k_{i}s_{i}}$ where $s_{i}\neq i$ and $k_{i}\neq q$. Then
$x_{k_{i}s_{i}}\notin C$ and $\{x_{k_{i}s_{i}}, x_{qi},\}$ is
in $E(G)$. There is at least two elements $i$ and $j$ such that $1\leq i< j\leq r$ and $s_{i}\neq s_{j}$, because $x_{qi}$ can not choose its adjacent vertex from the part $i$. Therefore the set $\{x_{k_{1}s_{1}}, \ldots , x_{k_{r}s_{r}}\}$ contain at least two elements. Then by hypothesis, at least two elements, say $a, b$ of $\{x_{k_{1}s_{1}}, \ldots , x_{k_{r}s_{r}}\}$ are adjacent by an edge. Now $C$ is a cover but $a, b$ are not in $C$, a contradiction.
\end{proof}

\begin{rem}
Villareal's theorem (Theorem 1.2) for bipartite graphs, and Haghighi's theorem (Theorem 1.3) for tripartite graphs, are  special cases of Theorem 2.3 (where $r=2$, and $r=3$).
\end{rem}

\section{\bf {\bf \em{\bf Examples and counterexamples}}}
\vskip 0.4 true cm

In this section, we give examples of two classes of unmixed graphs, and an example which shows that it is not necessary that an unmixed $r$-partite graph satisfies condition $(\ast )$.

\begin{example}
By Theorem 2.3, the following 4-partite graphs are unmixed.

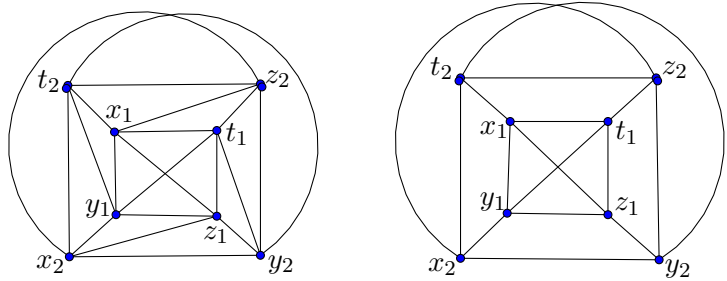
\definecolor{qqqqff}{rgb}{0.,0.,1.}
\begin{center}
\begin{tikzpicture}[line cap=round,line join=round,>=triangle 45,x=1.0cm,y=1.0cm]
\clip(-3.,1.) rectangle (7.,5.8);
\draw (5.86,4.26)-- (3.26,4.26);
\draw (3.26,4.26)-- (3.26,1.86);
\draw (5.86,4.26)-- (5.9,1.84);
\draw (5.9,1.84)-- (3.26,1.86);
\draw (3.92,3.68)-- (5.22,3.68);
\draw (5.22,3.68)-- (5.22,2.44);
\draw (5.22,2.44)-- (3.88,2.46);
\draw (3.92,3.68)-- (3.88,2.46);
\draw (5.22,2.44)-- (5.9,1.84);
\draw (5.22,3.68)-- (5.86,4.26);
\draw (3.92,3.68)-- (3.26,4.26);
\draw (3.88,2.46)-- (3.26,1.86);
\draw (3.92,3.68)-- (5.22,2.44);
\draw (5.22,3.68)-- (3.88,2.46);
\draw [shift={(4.2284210526315675,3.4192105263158026)}] plot[domain=0.4514648196483709:4.156601207784886,variable=\t]({1.*1.835477267784604*cos(\t r)+0.*1.835477267784604*sin(\t r)},{0.*1.835477267784604*cos(\t r)+1.*1.835477267784604*sin(\t r)});
\draw [shift={(4.91,3.41)}] plot[domain=-1.0081993171568522:2.6899926544431825,variable=\t]({1.*1.8560711193270627*cos(\t r)+0.*1.8560711193270627*sin(\t r)},{0.*1.8560711193270627*cos(\t r)+1.*1.8560711193270627*sin(\t r)});
\draw (0.6,1.9)-- (0.6,4.18);
\draw (0.6,4.18)-- (-1.96,4.16);
\draw (-1.96,4.16)-- (-1.94,1.88);
\draw (0.6,1.9)-- (-1.94,1.88);
\draw (-1.34,3.54)-- (0.02,3.56);
\draw (0.02,3.56)-- (0.02,2.42);
\draw (-1.34,3.54)-- (-1.32,2.44);
\draw (-1.32,2.44)-- (0.02,2.42);
\draw (0.02,2.42)-- (0.6,1.9);
\draw (0.02,3.56)-- (0.6,4.18);
\draw (-1.34,3.54)-- (-1.96,4.16);
\draw (-1.32,2.44)-- (-1.94,1.88);
\draw (0.6,1.9)-- (0.02,3.56);
\draw (0.6,4.18)-- (-1.34,3.54);
\draw (-1.96,4.16)-- (-1.32,2.44);
\draw (-1.94,1.88)-- (0.02,2.42);
\draw (-1.34,3.54)-- (0.02,2.42);
\draw (0.02,3.56)-- (-1.32,2.44);
\draw [shift={(-0.3952032520325195,3.35260162601626)}] plot[domain=-0.9701283912621932:2.6906444134786236,variable=\t]({1.*1.7608182747692012*cos(\t r)+0.*1.7608182747692012*sin(\t r)},{0.*1.7608182747692012*cos(\t r)+1.*1.7608182747692012*sin(\t r)});
\draw [shift={(-0.9754200542005421,3.367289972899729)}] plot[domain=0.4510333247708006:4.137034697971628,variable=\t]({1.*1.7726945408971515*cos(\t r)+0.*1.7726945408971515*sin(\t r)},{0.*1.7726945408971515*cos(\t r)+1.*1.7726945408971515*sin(\t r)});
\draw (-2.5,4.5) node[anchor=north west] {$t_{2}$};
\draw (-2.52,2.03) node[anchor=north west] {$x_{2}$};
\draw (0.56,2.04) node[anchor=north west] {$y_{2}$};
\draw (0.52,4.5) node[anchor=north west] {$z_{2}$};
\draw (-1.87,2.78) node[anchor=north west] {$y_{1}$};
\draw (-1.6,4.03) node[anchor=north west] {$x_{1}$};
\draw (-0.,3.8) node[anchor=north west] {$t_{1}$};
\draw (-0.3,2.42) node[anchor=north west] {$z_{1}$};
\draw (2.69,2.00) node[anchor=north west] {$x_{2}$};
\draw (5.84,1.98) node[anchor=north west] {$y_{2}$};
\draw (2.72,4.68) node[anchor=north west] {$t_{2}$};
\draw (5.8,4.62) node[anchor=north west] {$z_{2}$};
\draw (3.38,3.82) node[anchor=north west] {$x_{1}$};
\draw (3.38,2.87) node[anchor=north west] {$y_{1}$};
\draw (5.17,2.8) node[anchor=north west] {$z_{1}$};
\draw (5.18,3.82) node[anchor=north west] {$t_{1}$};
\begin{scriptsize}
\draw [fill=qqqqff] (5.86,4.26) circle (1.5pt);
\draw [fill=qqqqff] (3.26,4.26) circle (1.5pt);
\draw [fill=qqqqff] (3.26,1.86) circle (1.5pt);
\draw [fill=qqqqff] (5.9,1.84) circle (1.5pt);
\draw [fill=qqqqff] (3.92,3.68) circle (1.5pt);
\draw [fill=qqqqff] (5.22,3.68) circle (1.5pt);
\draw [fill=qqqqff] (5.22,2.44) circle (1.5pt);
\draw [fill=qqqqff] (3.88,2.46) circle (1.5pt);
\draw [fill=qqqqff] (5.88,4.22) circle (1.5pt);
\draw [fill=qqqqff] (3.24,4.22) circle (1.5pt);
\draw [fill=qqqqff] (0.6,1.9) circle (1.5pt);
\draw [fill=qqqqff] (0.6,4.18) circle (1.5pt);
\draw [fill=qqqqff] (-1.96,4.16) circle (1.5pt);
\draw [fill=qqqqff] (-1.94,1.88) circle (1.5pt);
\draw [fill=qqqqff] (-1.34,3.54) circle (1.5pt);
\draw [fill=qqqqff] (0.02,3.56) circle (1.5pt);
\draw [fill=qqqqff] (0.02,2.42) circle (1.5pt);
\draw [fill=qqqqff] (-1.32,2.44) circle (1.5pt);
\draw [fill=qqqqff] (-1.98,4.12) circle (1.5pt);
\draw [fill=qqqqff] (0.62,4.14) circle (1.5pt);
\end{scriptsize}
 \end{tikzpicture}
 \end{center}
 % \caption{}
 \label{graph 1}
% \end{figure}

\end{example}

In each of the above graphs, there are two complete graphs of order 4 and some edges between them.

For $r>4$, also $r=3$, using two complete graphs of order $r$, we can construct $r$-partite unmixed graphs which are natural generalization of the above graphs.

\smallskip

\begin{example}
For every $n$, $n\geq 3$, the complete graph $K_{n}$, is an $n$-partite graph which satisfies the condition $(\ast )$. By Theorem 2.3, $K_{n}$ is unmixed.
\end{example}

\smallskip

Theorem 2.3 dose not characterize all unmixed $r$-partite graphs.
More precisely, the condition $(\ast )$ is not valid for all unmixed graphs. In the following, we give an example of an
unmixed $r$-partite graph which dose not satisfy the condition
$(\ast )$.

\begin{example}
The following graph is a 4-partite graph with partition
$\{y_{1}\}$, $\{y_{2}, y_{4}\}$, $\{y_{3}\}$, and $\{y_{5},
y_{6}\}$. This graph dose not satisfy the condition $(\ast )$
because 6 is not a multiple of 4.

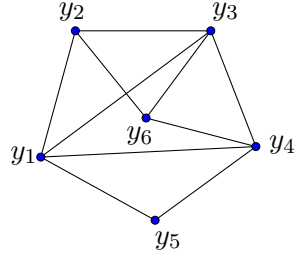
\definecolor{qqqqff}{rgb}{0.,0.,1.}
\begin{center}
\begin{tikzpicture}[line cap=round,line join=round,>=triangle 45,x=1.0cm,y=1.0cm]
\clip(1.3,0.7) rectangle (5.6,5.);
\draw (4.46,3.88)-- (2.66,3.88);
\draw (4.46,3.88)-- (5.06,2.34);
\draw (5.06,2.34)-- (3.72,1.36);
\draw (2.66,3.88)-- (2.2,2.2);
\draw (3.72,1.36)-- (2.2,2.2);
\draw (2.2,2.2)-- (4.46,3.88);
\draw (2.2,2.2)-- (5.06,2.34);
\draw (3.6,2.72)-- (4.46,3.88);
\draw (3.6,2.72)-- (5.06,2.34);
\draw (3.6,2.72)-- (2.66,3.88);
\draw (1.66,2.48) node[anchor=north west] {$y_{1}$};
\draw (2.30,4.4) node[anchor=north west] {$y_{2}$};
\draw (4.34,4.4) node[anchor=north west] {$y_{3}$};
\draw (5.10,2.62) node[anchor=north west] {$y_{4}$};
\draw (3.58,1.35) node[anchor=north west] {$y_{5}$};
\draw (3.20,2.74) node[anchor=north west] {$y_{6}$};
\begin{scriptsize}
\draw [fill=qqqqff] (2.66,3.88) circle (1.5pt);
\draw [fill=qqqqff] (4.46,3.88) circle (1.5pt);
\draw [fill=qqqqff] (5.06,2.34) circle (1.5pt);
\draw [fill=qqqqff] (3.72,1.36) circle (1.5pt);
\draw [fill=qqqqff] (2.2,2.2) circle (1.5pt);
\draw [fill=qqqqff] (3.6,2.72) circle (1.5pt);
\end{scriptsize}
 \end{tikzpicture}
 \end{center}
 % \caption{}
 \label{graph 1}
% \end{figure}

We show that this graph is unmixed. Let $C$ be an arbitrary
minimal vertex cover of $G$. We show that $C$ is of size 4.

Since $C$ is a cover, it selects at least one element of
$\{y_{4},y_{6}\}$. Now we consider the following cases:
\\\textbf{case 1:} $y_{6}\in C$ and $y_{4}\notin C$. In this case,
since $C$ is a vertex cover, $y_{1}, y_{3}, y_{5}\in C$. Now
$\{y_{1}, y_{3}, y_{5}, y_{6}\}$ is a vertex cover of $G$, and
since $C$ is minimal, $C=\{y_{1}, y_{3}, y_{5}, y_{6}\}$.
\\\textbf{case 2:} $y_{4}\in C$ and $y_{6}\notin C$. In this case, $y_{2},y_{3}\in C$, and  at least one
vertex of $y_{1},y_{5}$ and by minimality, only one is in $C$.
Now since $\{y_{2}, y_{3}, y_{4}, y_{i}\}$ where $i\in\{1, 5\}$
is one of two vertices $y_{1}$ and $y_{5}$, is a cover of $G$, by
minimality of $C$, $C=\{y_{2}, y_{3}, y_{4}, y_{i}\}$.
\\\textbf{case 3:} $y_{4}, y_{6}\in C$. In this case, at least one
of two vertices $y_{1}, y_{5}$ and by minimality of $C$, only one
is in $C$. Now if $y_{5}\in C$, $y_{3}$ should be in $C$ (because
the edge $\{y_{1}, y_{3}\}$ should be covered). Also $y_{2}\in C$
 (because the edge $\{y_{1}, y_{2}\}$ should be covered). Now
$\{y_{2}, y_{3}, y_{5}, y_{4}, y_{6}\}$ is a cover, and since $C$
is minimal, $C= \{y_{2}, y_{3}, y_{5}, y_{4}, y_{6}\}$, that is a
contradiction because $y_{6}$ can be removed. If $y_{1}\in C$, at
least one of $y_{2}$ and $y_{3}$, and by minimality only one, is
in $C$. Now since $\{y_{1}, y_{4}, y_{6}, y_{j}\}$, where
$j\in\{2, 3\}$ is one of two vertices $y_{2}$ and $y_{3}$, is a
vertex cover, by minimality of $C$, $C=\{y_{1}, y_{4}, y_{6},
y_{j}\}$.
\end{example}

%------------------------------------------------------------------------------------%

\vskip 0.4 true cm

%\begin{center}{\textbf{Acknowledgments}}
%\end{center}

%\vskip 0.4 true cm

%------------------------------------------------------------------------------------%


\begin{thebibliography}{20}
\bibitem{EE} M. Estrada and R. H. Villarreal, Cohen-Macaulay bipatite graphs, \textit{Arc. Math.}, \textbf{68} (1997), 124-128.
\bibitem{ET} O. Fanaron, Very well covered graphs, \textit{Discrete. Math.}, \textbf{42} (1982), no. 2-3, 177-187.
\bibitem{Ha} H. Haghighi, A generalization of Villarreal's result
for unmixed tripartite graphs, \textit{Bull. Iranian Math. Soc.},
\textbf{40} (2014), no. 6, 1505-1514.
\bibitem{EV} F. Harary, \textit{Graph Theory}, Addison-Wesley, Reading, MA,
1972.
\bibitem{HH} J. Herzog and T. Hibi, Distributive lattices, bipartite graphs, and Alexander duality, \textit{J. Algebraic Combin.}, \textbf{22} (2005), no. 3, 289-302.
\bibitem{K} R. M. Karp, Complexity of computer computation, \textit{Plenum Press},
New York, (1972), 85-103.
\bibitem{P} M. D. Plummer, Well-covered graphs: A survay, \textit{Questions Math.}, \textbf{16} (1993), no. 3, 253-287.
\bibitem{R} G. Ravindra, Well-covered graphs, \textit{J. Combin. Inform. }, System Sci.
\textbf{2} (1977), no. 1, 20-21.
\bibitem{V} R. H. Villarreal, Cohen-Macaulay graphs, \textit{Manuscripta
Math.}, \textbf{66} (1990), 277-293.
\bibitem{V1} R. H. Villarreal, \textit{Monomial Algebras}, Marcel Dekker, Inc. New York, 2001.
\bibitem{V2} R. H. Villarreal, Unmixed bipartite graphs, \textit{Rev. Colombiana
Mat.}, \textbf{41} (2007), no. 2, 393-395.
\bibitem{Z} R. Zaree-Nahandi, Pure simplicial
complexes and well-covered graphs, \textit{Rocky Mountain Journal
of Mathematics}, \textbf{45} (2015), no. 2, 695-702.
\end{thebibliography}
\end{document}